\documentclass{siamart0516}
\usepackage{euscript,amsmath,amssymb,amsfonts,graphicx,bm,mathtools}
  \usepackage{hyperref}
\usepackage{latexsym,amssymb,enumerate,amsmath} 
  \usepackage{graphicx} 
  \usepackage{tikz}
  \usepackage[square,sort,comma,numbers]{natbib}
\usepackage{pgfplots}
\usepackage{fancyvrb}

\numberwithin{equation}{section}
\theoremstyle{definition}
\newtheorem*{definition2*}{Definition}

\newcommand{\dd}{\textup{d}}
\def\eps{\varepsilon}
\def\E{\mathbb{E}}
\def\P{\mathbb{P}}
\def\R{\mathbb{R}}

\def\<{\langle}
\def\>{\rangle}
\newcommand{\drie}{\textcolor{black}{\mathcal{L}}_{\textup{rie}}}
\def\DD{\prescript{}{0}D_{t}^{1-\alpha}}
\def\D{\frac{\partial^{1-\alpha}}{\partial t^{1-\alpha}}}
\def\L{\textcolor{black}{\mathcal{L}}}
\usepackage{amsfonts}
\usepackage{graphicx}
\usepackage{epstopdf}
\usepackage{algorithmic}
\ifpdf
  \DeclareGraphicsExtensions{.eps,.pdf,.png,.jpg}
\else
  \DeclareGraphicsExtensions{.eps}
\fi

\newcommand{\TheTitle}{Cover times of many diffusive or subdiffusive searchers}
\newcommand{\ShortTitle}{Cover times of many diffusive searchers}
\newcommand{\TheAuthors}{Hyunjoong Kim and Sean D. Lawley}

\headers{\ShortTitle}{\TheAuthors}

\title{{\TheTitle}\thanks{%Submitted to the editors \today.
\funding{HK was supported by the Simons Foundation (Math + X grant) and SDL was supported by the National Science Foundation (CAREER DMS-1944574 and DMS-1814832).}}}

\author{Hyunjoong Kim\thanks{Center for Mathematical Biology \& Department of Mathematics,
University of Pennsylvania, Philadelphia, PA 19104, USA. HK was supported by the Simons Foundation (Math + X grant).} \and Sean D. Lawley\thanks{Department of Mathematics, University of Utah, Salt Lake City, UT 84112 USA (\texttt{lawley@math.utah.edu}). SDL was supported by the National Science Foundation (Grant Nos.\ CAREER DMS-1944574 and DMS-1814832).}
}
\date{\today}

\ifpdf
\hypersetup{
  pdftitle={\TheTitle},
  pdfauthor={\TheAuthors}
}
\fi

\begin{document}

\maketitle

\begin{abstract}
Cover times measure the speed of exhaustive searches which require the exploration of an entire spatial region(s). Applications include the immune system hunting pathogens, animals collecting food, robotic demining or cleaning, and computer search algorithms. Mathematically, a cover time is the first time a random searcher(s) comes within a specified ``detection radius'' of every point in the target region (often the entire spatial domain).
Due to their many applications and their fundamental probabilistic importance, cover times have been extensively studied in the physics and probability literatures. This prior work has generally studied cover times of a single searcher with a vanishing detection radius or a large target region. This prior work has further claimed that cover times for multiple searchers can be estimated by a simple rescaling of the cover time of a single searcher. In this paper, we study cover times of many diffusive or subdiffusive searchers and show that prior estimates break down as the number of searchers grows. We prove a rather universal formula for all the moments of such cover times in the many searcher limit that depends only on (i) the searcher's characteristic (sub)diffusivity and (ii) a certain geodesic distance between the searcher starting location(s) and the farthest point in the target. This formula is otherwise independent of the detection radius, space dimension, target size, and domain size. We illustrate our results in several examples and compare them to detailed stochastic simulations. 
\end{abstract}

% 60J60  	Diffusion processes 
% 35R11  	Fractional partial differential equations 
% 92C40  	Biochemistry, molecular biology

%\tableofcontents

%%%%%%%%%%%%%%%%%%%%%%%%%%%%%%%%%%%%%%%%%%%%%%%%%%%%%%%%%%%%%%%%%%%%%%%%%%%%%%%%%%%%%%%%%%%%%%%%%%%%%%%%%%%%%%%%%%%%%%%%%%%%%%%%%%%%%%%%%%%%%%%%%%%%%%%%%%%%%%%%%%%%%%%%%%%%%%%%%%%%%%%%%%%%%%%%%%%%%%%%%%%%%%%%%%%%%%%%%%%%%%
\section{Introduction}

How long does it take a random search process to find an entire collection of targets? This so-called \textit{cover time} measures the speed of exhaustive searches \cite{aldous1989}. Such exhaustive or comprehensive searches demand that the search process fully explore a region(s) (possibly the entire spatial domain), which contrasts searches that terminate upon the first hit to a target (whose speed is typically measured by a so-called first hitting time or first passage time \cite{redner2001, chou_first_2014}). Exhaustive searches are involved in a variety of applications \cite{chupeau2015}, including the immune system hunting pathogens (bacteria, viruses, etc.), animals collecting food and other resources, robots combing for dangers (chemical leaks, explosives, mines, etc.) or simply cleaning an area, and computer search algorithms.

There is a long history in both the physics literature \cite{nemirovsky1990, yokoi1990, brummelhuis1991, hemmer1998, nascimento2001, zlatanov2009, mendoncca2011, chupeau2014, chupeau2015, majumdar2016, grassberger2017, cheng2018, maziya2020, dong2023, han2023} and the probability literature \cite{aldous1983, aldous1989, broder1989, aldous1989b, kahn1989, dembo2003, dembo2004, ding2012, belius2013, belius2017} analyzing cover times. To describe more precisely, let $X=\{X(t)\}_{t\ge0}$ denote the path of a searcher randomly exploring a $d$-dimensional spatial domain $M\subseteq \R^{d}$. If the searcher has some detection radius $r>0$, then the region explored by the searcher by time $t\ge0$ is
\begin{align}\label{wiener}
\mathcal{S}(t)
:=\cup_{s=0}^{t}B(X(s),r)\subseteq M,
\end{align}
where $B(x,r)$ denotes the ball of radius $r$ centered at $x\in M$ (in the case that $X$ is a Brownian motion, \eqref{wiener} is called the Wiener sausage \cite{donsker1975}). The cover time of a given target region $U_{\text{T}}\subseteq M$ is then
\begin{align}\label{singlesigma}
\sigma
:=\inf\{t>0:U_{\text{T}}\subseteq \mathcal{S}(t)\}.
\end{align}

Prior work has generally studied the statistics and distribution of the cover time $\sigma$ in the case that the detection radius $r$ is much smaller than the size of the target region $U_{\text{T}}$, which is often taken to be the entire domain (i.e.\ $U_{\text{T}}=M$). For example, suppose $X$ is a Brownian motion with diffusivity $D>0$ on a smooth, compact, connected, $d$-dimensional Riemannian manifold without boundary. In the limit of a vanishing detection radius, it was proven that the cover time $\sigma$ diverges according to
\begin{align}\label{blowup1}
\sigma
\sim t_{d}\quad\text{almost surely as $r\to0$},
\end{align}
where the divergence of the deterministic time $t_{d}$ depends critically on the space dimension,
\begin{align}\label{blowup2}
t_{d}
=\begin{dcases}
\frac{1}{\pi}\frac{|M|}{D}(\ln(1/r))^{2}& \text{if }d=2,\\
\frac{d\,\Gamma(d/2)}{2(d-2)\pi^{d/2}}\frac{|M|r^{2-d}}{D}\ln(1/r) & \text{if }d\ge3,
\end{dcases}
\end{align}
where $|M|$ denotes the $d$-dimensional volume of the spatial domain. 
The results in \eqref{blowup1}-\eqref{blowup2} can be derived heuristically (see \cite{dembo2003}) by combining (i) asymptotics of the principal eigenvalue of the Laplacian in a domain with a small hole \cite{ozawa1983, ward1993} with (ii) classical results in extreme value theory \cite{haanbook}. Indeed, it was shown in the physics literature \cite{chupeau2015} that cover times for non-compact random walks on discrete state spaces follow an extreme value distribution (specifically, Gumbel) with a diverging mean in the limit of a large target region.

Equations~\eqref{wiener}-\eqref{singlesigma} and \eqref{blowup1}-\eqref{blowup2} concern a single searcher, but the applications mentioned above generally involve cover times of multiple searchers. To explain, consider $N\ge1$ independent and identically distributed (iid) random searchers, $\{X_{n}\}_{n=1}^{N}$, and let $\mathcal{S}_{N}(t)$ denote the region explored by this collection of searchers by time $t\ge0$,
\begin{align}\label{setN}
\mathcal{S}_{N}(t)
:=\cup_{n=1}^{N}\cup_{s=0}^{t}B(X_{n}(s),r)\subseteq M.
\end{align}
The cover time of a target $U_{\text{T}}$ by these $N$ searchers is then
\begin{align*}
\sigma_{N}
:=\inf\{t>0:U_{\text{T}}\subseteq \mathcal{S}_{N}(t)\}.
\end{align*}
See Figure~\ref{figschem} for an illustration. 
Prior works have argued that the cover time for $N$ searchers can be obtained by simply rescaling the cover time of a single searcher \cite{chupeau2015, dong2023}. Specifically, these works have claimed that as long as $N$ is not too large,
\begin{align}\label{approxe}
\sigma_{N}
\approx_{\text{dist}}
\sigma/N,
\end{align}
where $\approx_{\text{dist}}$ denotes approximate equality in distribution.

%%%%%%%%%%%%%%%%%%%%%%%%%%%%%%%%%
\begin{figure}
  \centering
             \includegraphics[width=0.6\textwidth]{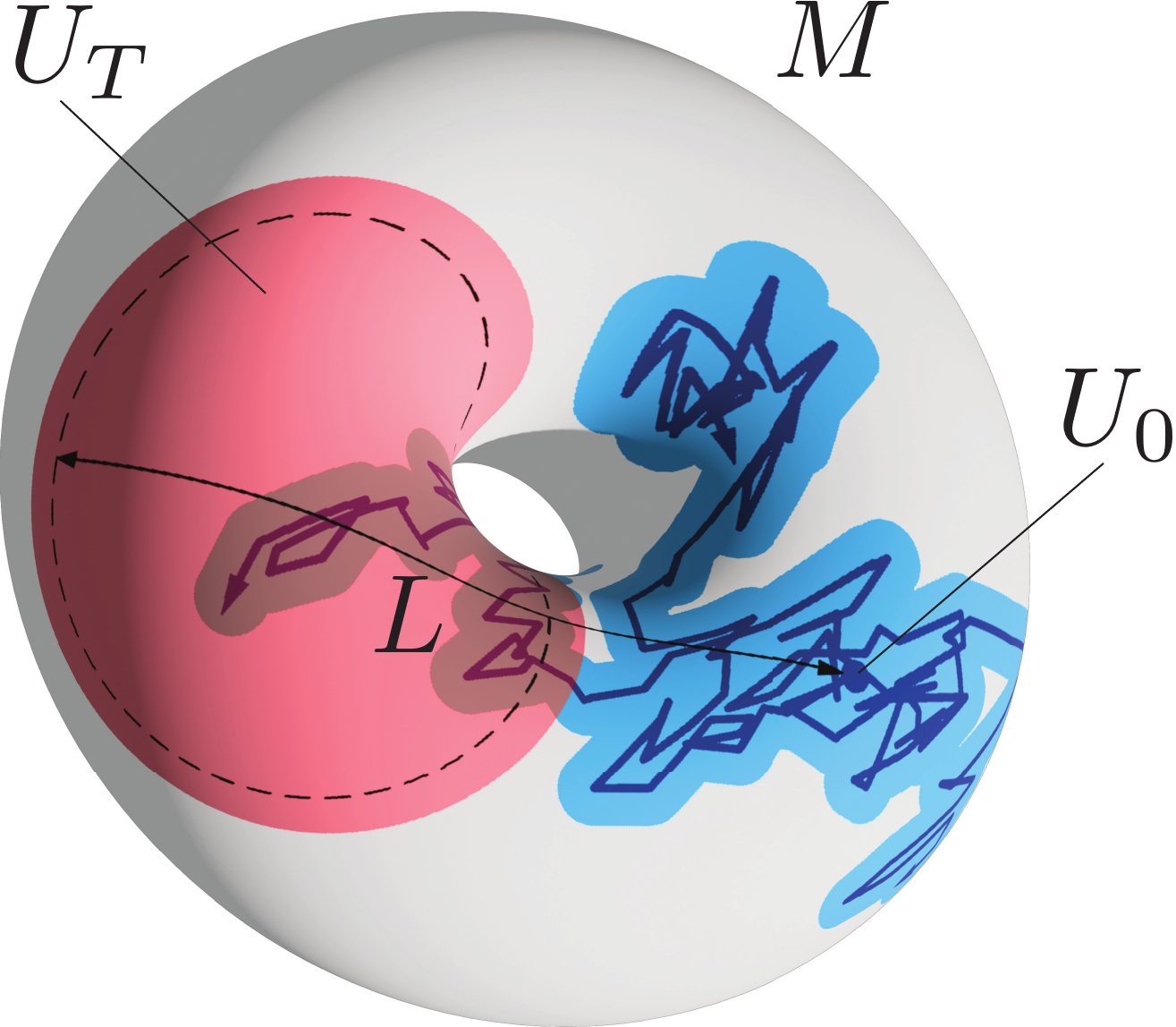}                   
    \caption{Example of a spatial domain $M$ given by the $(d=2)$-dimensional torus with $N=3$ diffusive searchers that all start at the single point labeled $U_{0}$. The dark blue indicates the paths of the searchers and the light blue region indicates the corresponding region detected (the set $S_{N}(t)$ in \eqref{setN}). The red region labeled $U_{\text{T}}$ denotes the target and the black curve labeled $L$ depicts the shortest path a searcher must travel to cover the farthest part of the target.}
 \label{figschem}
\end{figure}
%%%%%%%%%%%%%%%%%%%%%%%%%%%%%%%%%

In this paper, we analyze the cover time of $N\gg1$ diffusive or subdiffusive searchers. We find that the approximation \eqref{approxe} breaks down as $N$ grows, as we prove the following rather universal formula for the $m$th moment of the cover time for $N$ diffusive searchers,
\begin{align}\label{main0}
\E[(\sigma_{N})^{m}]
\sim\Big(\frac{L^{2}}{4D\ln N}\Big)^{m}\quad\text{as }N\to\infty,
\end{align}
for any moment $m\ge1$. In \eqref{main0}, $D>0$ is a searcher's characteristic diffusivity, and $L>0$ is the distance a searcher must travel to reach the farthest part of the target. More precisely, 
\begin{align}\label{L0}
L
=\sup_{y\in U_{\text{T}}}\L(U_{0},B(y,r))>0,
\end{align}
where $U_{0}\subset M$ is the support of the initial positions of the searchers and $\L:M\times M\to[0,\infty)$ is a certain distance metric depending on the geometry of the space $M$ (and also on the stochastic dynamics of the searchers in the case of a space-dependent diffusivity, see section~\ref{exdrift}). The formula in \eqref{main0} means that the $1/N$ decay in \eqref{approxe} slows down markedly to $1/\ln N$ for large $N$. Hence, the cover time ``speed-up'' gained from additional searchers saturates for many searchers. Put another way, halving the cover time requires merely doubling the number of searchers in the regime in \eqref{approxe}, whereas halving the cover time requires squaring the number of searchers in the large $N$ regime of \eqref{main0}.

To illustrate \eqref{main0} in a simple example, suppose the searchers move by pure Brownian motion with diffusivity $D$, the spatial domain $M$ is the $d$-dimensional torus with diameter $l=\sup_{x,y\in M}\|x-y\|$ with $l>r$, the target is the entire torus $U_{\text{T}}=M$, and the searchers all start at a single point $U_{0}=x_{0}\in M$. In this case, \eqref{main0} holds with $L=l-r>0$. Our proof of \eqref{main0} relies on large deviation theory and short-time asymptotics of the heat kernel \cite{varadhan1967, norris1997}, extreme first passage time theory \cite{lawley2020uni}, and combinatorial arguments using the inclusion-exclusion principle \cite{durrett2019}. We note that \eqref{main0} was previously shown for the time it takes one-dimensional, pure Brownian motions to cover an interval with either periodic or reflecting boundary conditions \cite{majumdar2016}.

There are several counterintuitive features of \eqref{main0}. First, \eqref{main0} is independent of the space dimension $d\ge1$. This contrasts the cover time for a single searcher (i.e.\ $N=1$), which, as one expects, is slower for higher space dimensions $d$. Second, \eqref{main0} depends only weakly on the detection radius $r$. In particular, if we fix any detection radius $r>0$, then \eqref{main0} implies that the $m$th moment of the cover time has the following upper bound for sufficiently large $N$,
\begin{align}\label{cub0}
\E[(\sigma_{N})^{m}]
\le
\Big(\frac{(L_{+})^{2}}{4D\ln N}\Big)^{m},
\end{align}
where $L_{+}=\sup_{y\in U_{\text{T}}}\L(U_{0},y)>L$. The upper bound \eqref{cub0} is perhaps counterintuitive since the cover time diverges as the detection radius vanishes. Naturally, a smaller detection radius $r$ requires a larger searcher number $N$ in order for \eqref{cub0} to hold. In the Discussion section, we address the question of when a system is in the large $N$ regime versus the small detection radius regime.

Third, \eqref{main0} depends on the geometry of the target set $U_{\text{T}}$ only through \eqref{L0}. In particular, \eqref{main0} is unchanged if the target set is reduced to only a single point achieving the supremum in \eqref{L0}. 
That is, the cover time for a large target is no longer than the cover time of this single point in the many searcher limit. In addition, the cover time is unaffected by changing the domain $M$ (assuming $L$ in \eqref{L0} is unchanged). This contrasts with the cover time of a single searcher which is proportional to the volume of the domain $M$ for a small detection radius (see \eqref{blowup1}-\eqref{blowup2}). Finally, we find in section~\ref{exdrift} that \eqref{main0} is independent of any deterministic drift biasing the motion of the searchers.

We also determine the moments of the cover time for many subdiffusive searchers. Subdiffusion is marked by a mean-squared displacement that grows according to a sublinear power law. Specifically, if $\{Y(t)\}_{t\ge0}$ denotes the path of a subdiffusive searcher, then 
\begin{align*}
\E\|Y(t)-Y(0)\|^{2}\propto t^{\alpha},\quad \alpha\in(0,1),
\end{align*}
and $\alpha\in(0,1)$ is termed the subdiffusive exponent. Subdiffusion has been observed in many physical systems \cite{oliveira2019, klafter2005, barkai2012, sokolov2012} and is especially prevalent in cell biology \cite{golding2006, hofling2013}. Assuming that the subdiffusion is modeled by a time fractional Fokker-Planck equation \cite{metzler1999} (which is equivalent to an appropriate random time change of a normal diffusive process, see section~\ref{exsub}), we prove the following formula for the $m$th moment of the cover time for $N$ subdiffusive searchers,
\begin{align}\label{main0sub}
\E[(\sigma_{N})^{m}]
\sim
\Big(\frac{\alpha(2-\alpha)^{\frac{2-\alpha}{\alpha}}\big(L^{2}/(4{{D}})\big)^{1/\alpha}}{(\ln N)^{2/\alpha-1}}\Big)^{m}\quad\text{as }N\to\infty,
\end{align}
where $\alpha\in(0,1)$ is the subdiffusion exponent, ${{D}}>0$ is a characteristic subdiffusion coefficient (with dimension $(\text{length})^{2}(\text{time})^{-\alpha}$), and the length $L$ is in \eqref{L0}. Notably, comparing \eqref{main0sub} with \eqref{main0} implies that many subdiffusive searchers cover a target faster than many diffusive searchers (i.e.\ the cover time in \eqref{main0sub} is less than the cover time in \eqref{main0} for sufficiently large $N$).

The rest of the paper is organized as follows. In section~\ref{math}, we prove \eqref{main0} and \eqref{main0sub} in a general mathematical setting. In section~\ref{examples}, we illustrate these general results in several specific examples and compare to stochastic simulations. 
We conclude by discussing related work, including addressing the question of when a system is in the small detection radius regime of \eqref{blowup1}-\eqref{blowup2} versus the many searcher regime of \eqref{main0}. 

%%%%%%%%%%%%%%%%%%%%%%%%%%%%%%%%%%%%%%%%%%%%%%%%%%%%%%%%%%%%%%%%%%%%%%%%%%%%%%%%%%%%%%%%%%%%%%%%%%%%%%%%%%%%%%%%%%%%%%%%%%%%%%%%%%%%%%%%%%%%%%%%%%%%%%%%%%%%%%%%%%%%%%%%%%%%%%%%%%%%%%%%%%%%%%%%%%%%%%%%%%%%%%%%%%%%%%%%%%%%%%
\section{General mathematical analysis}\label{math}

In this section, we prove the cover time moment formulas in \eqref{main0} and \eqref{main0sub} for diffusive and subdiffusive search under generic assumptions on the first passage times of single searchers.

%%%%%%%%%%%%%%%%%%%%%%%%%%%%%%%%%%%%%%%%%%%%%%%%%%%%%%%%%%%%%%%%%%%%%%%%%%%%%%%%%%%%%%%%%%%%%%%%%%%%%%%%%%%%
\subsection{Setup and main theorem for diffusive search}\label{setup}

Consider $N\ge1$ iid searchers, $\{X_{n}(t)\}_{t\ge0}$ for $n=1,\dots,N$, on a set $M\subset\R^{d}$. For a given metric ${\L_{B}}$ on $M$, let $B(x,r)$ denote the closed ball of radius $r>0$ centered at $x\in M$,
\begin{align}\label{ball}
B(x,r)
:=\{y\in M:{\L_{B}}(x,y)\le r\}.
\end{align}
Consider the $d$-dimensional region detected by the $N$ searchers by time $t\ge0$,
\begin{align*}
\mathcal{S}_{N}(t)
:=\cup_{n=1}^{N}\cup_{s=0}^{t}B(X_{n}(s),r)\subseteq M.
\end{align*}
For some compact target set $U_{\text{T}}\subseteq M$, define the cover time
\begin{align*}
\sigma_{N}
:=\inf\{t>0:U_{\text{T}}\subseteq \mathcal{S}_{N}(t)\}.
\end{align*}
Assume the initial distribution of the searchers has compact support $U_{0}\subset M$ which satisfies
\begin{align}\label{nontrivial0}
\sup_{y\in U_{\text{T}}}{\L_{B}}(U_{0},B(y,r))>0,
\end{align}
where we define the distance between any two sets $U,V\subseteq M$ by
\begin{align}\label{setdistance}
{\L_{B}}(U,V)
:=\inf_{x\in U,y\in V}{\L_{B}}(x,y).
\end{align}
The assumption~\eqref{nontrivial0} precludes the trivial case that the searchers immediately cover the target due to their initial placement (i.e.\ $\sigma_{N}=0$).

For any set $U\subset M$, let $\tau(U)$ denote the FPT to $U$ for a single searcher,
\begin{align}\label{tauU}
\tau(U)
:=\inf\{t>0:X_{1}(t)\in U\}.
\end{align}
Assume there exists a diffusivity $D>0$ and a metric $\L$ on $M$ that is bilipschitz equivalent\footnote{A pair of metrics $\L$ and ${\L_{B}}$ on $M$ are bilipschitz equivalent if there exists constants $\beta_{1}>0$, $\beta_{2}>0$ such that $\beta_{1}\L(x,y)\le {\L_{B}}(x,y)\le \beta_{2}\L(x,y)$ for all $x,y\in M$.} to ${\L_{B}}$ such that for any set $U\subset M$ equal to a finite union of balls of the form
\begin{align*}
U=\cup_{i=1}^{k}B(x_{i},r'),
\end{align*}
for $r'>0$ and $x_{1},\dots,x_{k}\in U_{\text{T}}$ with $1\le k<\infty$, we have that
\begin{align}\label{finite}
\int_{0}^{\infty} (\P(\tau(U)> t))^{N}\,\dd t<\infty\quad\text{for some }N\ge1,
\end{align}
and
\begin{align}\label{log}
\lim_{t\to0+}t\ln \P(\tau(U)\le t)
=-\frac{\L^{2}(U_{0},U)}{4D}\le0,
\end{align}
where $\L(U_{0},U)$ is defined analogously to \eqref{setdistance}.

Before stating our main theorem, we comment on assumptions \eqref{finite} and \eqref{log}. The integral in \eqref{finite} is the mean fastest first passage time out of $N$ searchers, i.e.\ $\E[\min\{\tau_{1}(U),\dots,\tau_{N}(U)\}]$ where $\tau_{1}(U),\dots,\tau_{N}(U)$ are $N$ iid realizations of $\tau(U)$ in \eqref{tauU}. Hence, assumption \eqref{finite} requires that this mean fastest first passage time is finite for sufficiently large $N$. The simplest way for \eqref{finite} to hold is simply that $\E[\tau(U)]<\infty$ (i.e.\ \eqref{finite} holds for all $N\ge1$), which is typically the case for diffusion in a bounded spatial domain. More generally, \eqref{finite} is assured as long as $\P(\tau(U)>t)$ decays no slower than a power law as $t\to\infty$. Assumption~\eqref{log} concerns the short-time behavior of the distribution of first passage times. As we review in section~\ref{examples}, the particular short-time behavior in \eqref{log} is characteristic of diffusion in very general mathematical scenarios and is often referred to as Varadhan's formula in the large deviation literature \cite{varadhan1967, norris1997}.

%%%%%%%%%%%%%%%%%%%%%%%%%%%%%%%%%%%%%%
\begin{theorem}\label{main}
Under the assumptions of \eqref{nontrivial0}, \eqref{finite}, and \eqref{log}, for any moment $m\ge1$ we have that
\begin{align}\label{mainformula}
\E[(\sigma_{N})^{m}]
\sim
\Big(\frac{L^{2}}{4D\ln N}\Big)^{m}\quad\text{as }N\to\infty,
\end{align}
where
\begin{align*}
L=\sup_{y\in U_{\text{T}}}\L(U_{0},B(y,r))>0.
\end{align*}
\end{theorem}
%%%%%%%%%%%%%%%%%%%%%%%%%%%%%%%%%%%%%%

An immediate corollary of Theorem~\ref{main} is that the cover time becomes deterministic for many searchers in the sense that its coefficient of variation vanishes,
\begin{align}\label{cv}
\frac{\sqrt{\textup{Variance}(\sigma_{N})}}{\E[\sigma_{N}]}
\to0\quad\text{as }N\to\infty.
\end{align}

%%%%%%%%%%%%%%%%%%%%%%%%%%%%%%%%%%%%%%%%%%%%%%%%%%%%%%%%%%%%%%%%%%%%%%%%%%%%%%%%%%%%%%%%%%%%%%%%%%%%%%%%%%%%%%%%%%%%%%%%%%%%%%%%%%%%%%%%%%%%%%%%%%%%%%%%%%%%%%%%%%%%%%%%%%%%%%%%%%%%%%%%%%%%%%%%%%%%%%%%%%%%%%%%%%%%%%%%%%%%%%
\subsection{Extension to subdiffusive search}

The moment formula in Theorem~\ref{main} can be quickly extended to subdiffusive search. As we discuss in section~\ref{exsub}, the short-time behavior of the FPT of a subdiffusive searcher modeled by a fractional Fokker-Planck equation is generally given by a minor modification of \eqref{log}. In particular, \eqref{log} is replaced by
\begin{align}\label{logsub}
\lim_{t\to0+}t^{\frac{\alpha }{2-\alpha}}\ln \P(\tau(U)\le t)
=-(2-\alpha) \alpha ^{\frac{\alpha }{2-\alpha}} \Big(\frac{\L^{2}(U_{0},U)}{4{{D}}}\Big)^{\frac{1}{2-\alpha }}\le0,
\end{align}
where $\alpha\in(0,1)$ is the subdiffusion exponent and ${{D}}>0$ a characteristic subdiffusion coefficient (with dimension $(\text{length})^{2}(\text{time})^{-\alpha}$).

%%%%%%%%%%%%%%%%%%%%%%%%%%%%%%%%%%%%%%
\begin{theorem}\label{mainsub}
Under the assumptions of \eqref{nontrivial0}, \eqref{finite}, and \eqref{logsub}, for any moment $m\ge1$ we have that
\begin{align}\label{mainformulasub}
\E[(\sigma_{N})^{m}]
\sim
\bigg(\frac{\alpha(2-\alpha)^{\frac{2-\alpha}{\alpha}}\big(L^{2}/(4{{D}})\big)^{1/\alpha}}{(\ln N)^{2/\alpha-1}}\bigg)^{m}\quad\text{as }N\to\infty,
\end{align}
where
\begin{align*}
L=\sup_{y\in U_{\text{T}}}\L(U_{0},B(y,r))>0.
\end{align*}
\end{theorem}
%%%%%%%%%%%%%%%%%%%%%%%%%%%%%%%%%%%%%%

Observe that the vanishing coefficient of variation in \eqref{cv} is also implied by Theorem~\ref{mainsub}. Observe also that Theorem~\ref{mainsub} reduces to Theorem~\ref{main} if $\alpha=1$.

%%%%%%%%%%%%%%%%%%%%%%%%%%%%%%%%%%%%%%%%%%%%%%%%%%%%%%%%%%%%%%%%%%%%%%%%%%%%%%%%%%%%%%%%%%%%%%%%%%%%%%%%%%%%%%%%%%%%%%%%%%%%%%%%%%%%%%%%%%%%%%%%%%%%%%%%%%%%%%%%%%%%%%%%%%%%%%%%%%%%%%%%%%%%%%%%%%%%%%%%%%%%%%%%%%%%%%%%%%%%%%
\subsection{Proofs}

\begin{proof}[Proof of Theorem~\ref{main}]

For each $x\in U_{\text{T}}$ and $r>0$, let $\tau_{n}(B(x,r))$ be the FPT of the $n$th searcher to the ball $B(x,r)$ defined in \eqref{ball},
\begin{align*}
\tau_{n}(B(x,r))
:=\inf\{t>0:X_{n}(t)\in B(x,r)\},\quad n\in\{1,\dots,N\}.
\end{align*}
Let $T_{N}(x,r)$ denote the corresponding fastest FPT,
\begin{align*}
T_{N}(x,r)
:=\min\{\tau_{1}(B(x,r)),\dots,\tau_{N}(B(x,r))\}.
\end{align*}
The survival probability of $\sigma_{N}$ can be written in terms of these fastest FPTs,
\begin{align}\label{key}
\P(\sigma_{N}>t)
=\P(\cup_{x\in U_{\text{T}}}\{T_{N}(x,r)>t\}).
\end{align}

The representation \eqref{key} immediately yields the lower bound,
\begin{align}\label{easybound}
\P(\sigma_{N}>t)
\ge\P(T_{N}(y,r)>t),\quad\text{for any }y\in U_{\text{T}}.
\end{align}
Since for any nonnegative random variable $T\ge0$ and any moment $m\ge1$ we have \cite{durrett2019}
\begin{align}\label{useful}
\E[T^{m}]
=\int_{0}^{\infty}\P(T>z^{1/m})\,\dd z,
\end{align}
replacing $t$ by $z^{1/m}$ and integrating the bound \eqref{easybound} over $z\ge0$ implies that 
\begin{align*}
\E[(\sigma_{N})^{m}]
\ge\E[(T_{N}(y,r))^{m}],\quad\text{for any }y\in U_{\text{T}}.
\end{align*}
Hence, the assumptions in \eqref{finite}-\eqref{log} and Theorem~1 in \cite{lawley2020uni} imply
\begin{align}\label{sim1}
\liminf_{N\to\infty}\E[(\sigma_{N})^{m}](\ln N)^{m}
\ge \Big(\frac{\L^{2}(U_{0},B(y,r))}{4D}\Big)^{m}
\ge0,\quad\text{for any }y\in U_{\text{T}}.
\end{align}
Taking the supremum over $y\in U_{\text{T}}$ yields the best lower bound,
\begin{align}\label{tim1}
\liminf_{N\to\infty}\E[(\sigma_{N})^{m}](\ln N)^{m}
\ge \Big(\sup_{y\in U_{\text{T}}}\frac{\L^{2}(U_{0},B(y,r))}{4D}\Big)^{m}
>0,
\end{align}
where the final inequality is strict due to \eqref{nontrivial0} and the bilipschitz equivalence of $\L$ and ${\L_{B}}$.

To get an upper bound on the large $N$ behavior of $\E[(\sigma_{N})^{m}]$, let $\eps\in(0,1)$ be arbitrary. Since the target $U_{\text{T}}$ is compact, we may choose a finite set of points $x_{1},x_{2},\dots,x_{k}\in U_{\text{T}}$ so that $\cup_{i=1}^{k}B(x_{k},\eps r)$ covers $U_{\text{T}}$,
\begin{align}\label{epscover}
U_{\text{T}}\subset\cup_{i=1}^{k}B(x_{i},\eps r).
\end{align}
We claim that for any $i\in\{1,\dots,k\}$,
\begin{align}\label{imm}
B(x_{i},\eps r)\subset B(y_{i},r)\quad\text{if }y_{i}\in B(x_{i},(1-\eps)r).
\end{align}
To see \eqref{imm}, assume $y_{i}\in B(x_{i},(1-\eps)r)$ and $z\in B(x_{i},\eps r)$ and observe that the triangle inequality implies
\begin{align*}
\L_{B}(z,y_{i})
\le \L_{B}(z,x_{i})+\L_{B}(x_{i},y_{i})
\le\eps r+(1-\eps)r=r,
\end{align*}
and thus $z\in B(y_{i},r)$, which proves \eqref{imm}. Hence, \eqref{key} implies
\begin{align}\label{ub}
\begin{split}
\P(\sigma_{N}>t)
&=\P(\cup_{x\in U_{\text{T}}}\{T_{N}(x,r)>t\})\\
&\le\P(\cup_{i=1}^{k}\{T_{N}(x_{i},(1-\eps)r)>t\}).
\end{split}
\end{align}
To see this, suppose that there is an $x\in U_{\text{T}}$ such that $T_{N}(x,r)>t$. By \eqref{epscover}, there exists an $x_{i}\in U_{\text{T}}$ such that ${\L_{B}}(x_{i},x)\le\eps r$ and thus $B(x_{i},(1-\eps)r)\subset B(x,r)$. Hence, $T_{N}(x_{i},(1-\eps)r)>t$ which proves \eqref{ub}.

To estimate the upper bound in \eqref{ub}, we use the inclusion-exclusion principle to obtain \cite{durrett2019}
\begin{align*}
&\P(\cup_{i=1}^{k}\{T_{N}(x_{i},(1-\eps)r)>t\})\\
&\quad=\sum_{j=1}^{k} \Big((-1)^{j-1}\sum_{I\subseteq\{1,\ldots,k\}\atop |I|=j} \mathbb{P}(\cap_{i\in I}\{T_{N}(x_{i},(1-\eps)r)>t\})\Big)\\
&\quad=\sum_{j=1}^{k} \Big((-1)^{j-1}\sum_{I\subseteq\{1,\ldots,k\}\atop |I|=j} \mathbb{P}(\min_{i\in I}\{T_{N}(x_{i},(1-\eps)r)\}>t)\Big),
\end{align*}
where inner sum runs over all subsets $I$ of the indices $1,\dots,k$ which contain exactly $j$ elements. Replacing $t$ by $z^{1/m}$ and integrating over $z\ge0$ yields
\begin{align*}
&\sum_{j=1}^{k} \Big((-1)^{j-1}\sum_{I\subseteq\{1,\ldots,k\}\atop |I|=j} \int_{0}^{\infty}\mathbb{P}(\min_{i\in I}\{T_{N}(x_{i},(1-\eps)r)\}>z^{1/m})\,\dd z\Big)\\
&=\sum_{j=1}^{k} \Big((-1)^{j-1}\sum_{I\subseteq\{1,\ldots,k\}\atop |I|=j} \E\Big[\big(\min_{i\in I}\{T_{N}(x_{i},(1-\eps)r)\}\big)^{m}\Big]\Big).
\end{align*}
Further, the assumptions in \eqref{finite}-\eqref{log} and Theorem~1 in \cite{lawley2020uni} imply
\begin{align}\label{sim2}
\begin{split}
&\lim_{N\to\infty}(\ln N)^{m}\E\Big[\big(\min_{i\in I}\{T_{N}(x_{i},(1-\eps)r)\}\big)^{m}\Big]\\
&\quad=\Big(\frac{\min_{i\in I}\L^{2}(U_{0},B(x_{i},(1-\eps)r)}{4D}\Big)^{m}\ge0,
\end{split}
\end{align}

Therefore, \eqref{ub} implies
\begin{align}
\begin{split}\label{tim2}
&\limsup_{N\to\infty} (\ln N)^{m}\E[(\sigma_{N})^{m}]\\
&\quad\le\sum_{j=1}^{k}\Big((-1)^{j-1}\sum_{I\subseteq\{1,\ldots,k\}\atop |I|=j}\Big(\frac{\min_{i\in I}\L^{2}(U_{0},B(x_{i},(1-\eps)r)}{4D}\Big)^{m}\Big)\\
&\quad=\max_{i\in\{1,\dots,k\}}\Big(\frac{\L^{2}(U_{0},B(x_{i},(1-\eps)r)}{4D}\Big)^{m},
\end{split}
\end{align}
where the equality is due to Lemma~\ref{comb}. below. It follows from \eqref{epscover} that for any $\eta>0$, we may choose $\eps$ sufficiently small so that 
\begin{align*}
\max_{i\in\{1,\dots,k\}}\Big(\frac{\L^{2}(U_{0},B(x_{i},(1-\eps)r)}{4D}\Big)^{m}
\le\Big(\sup_{y\in U_{\text{T}}}\frac{\L^{2}(U_{0},B(y,r))}{4D}\Big)^{m}+\eta.
\end{align*}
Since $\eta>0$ is arbitrary, the proof is complete.
\end{proof}

%%%%%%%%%%%%%%%%%%%%%%%%%%%%%%%%%%%%%%%%%%%%%%%%%%%%%%%%%%%%%%%%%%%%%%%%%%%%%%%%%%%%%%%%%%%%%%%%%%%%%%%%%%%%%%%%%%%%%%%%

\begin{proof}[Proof of Theorem~\ref{mainsub}]
The proof of Theorem~\ref{mainsub} follows identical steps as the proof of Theorem~\ref{main}, except for a few differences. The first difference is that, due to an application of Theorem~6 in \cite{lawley2020sub}, the inequality \eqref{sim1} is replaced by the following inequality for any $y\in U_{\text{T}}$,
\begin{align*}
\liminf_{N\to\infty}\E[(\sigma_{N})^{m}](\ln N)^{m(2/\alpha-1)}
\ge \Big(\theta\Big[\L^{2}(U_{0},B(y,r))\Big]^{\frac{1}{2-\alpha}}\Big)^{m(2/\alpha-1)}
\ge0,
\end{align*}
where $\theta=(2-\alpha)\alpha^{\frac{\alpha}{2-\alpha}}(4{{D}})^{-1/(2-\alpha)}$. Taking the supremum over $y\in U_{\text{T}}$ then yields the following lower bound in place of \eqref{tim1},
\begin{align*}
\liminf_{N\to\infty}\E[(\sigma_{N})^{m}](\ln N)^{m(2/\alpha-1)}
&\ge \Big(\theta\sup_{y\in U_{\text{T}}}\Big[\L^{2}(U_{0},B(y,r))\Big]^{\frac{1}{2-\alpha}}\Big)^{m(2/\alpha-1)}\\
&= \Big(\theta^{\frac{2-\alpha}{\alpha}}\sup_{y\in U_{\text{T}}}\Big[\L^{2}(U_{0},B(y,r))\Big]^{1/\alpha}\Big)^{m}
>0.
\end{align*}

The next difference is that, due to an application of Theorem~6 in \cite{lawley2020sub}, equation~\eqref{sim2} is replaced by
\begin{align*}
&\lim_{N\to\infty}(\ln N)^{m(2/\alpha-1)}\E\Big[\big(\min_{i\in I}\{T_{N}(x_{i},(1-\eps)r)\}\big)^{m}\Big]\\
&\quad=\Big(\theta^{\frac{2-\alpha}{\alpha}}\Big[\min_{i\in I}\L^{2}(U_{0},B(x_{i},(1-\eps)r)\Big]^{1/\alpha}\Big)^{m}\ge0,
\end{align*}
and thus \eqref{tim2} is replaced by
\begin{align*}
&\limsup_{N\to\infty} (\ln N)^{m(2/\alpha-1)}\E[(\sigma_{N})^{m}]\\
&\quad\le\sum_{j=1}^{k}\Big((-1)^{j-1}\sum_{I\subseteq\{1,\ldots,k\}\atop |I|=j}\Big(\theta^{\frac{2-\alpha}{\alpha}}\Big[\min_{i\in I}\L^{2}(U_{0},B(x_{i},(1-\eps)r)\Big]^{1/\alpha}\Big)^{m}\Big)\\
&\quad=\max_{i\in\{1,\dots,k\}}\Big(\theta^{\frac{2-\alpha}{\alpha}}\Big[\L^{2}(U_{0},B(x_{i},(1-\eps)r)\Big]^{1/\alpha}\Big)^{m},
\end{align*}
where the equality is due to Lemma~\ref{comb} below. It follows from \eqref{epscover} that for any $\eta>0$, we may choose $\eps$ sufficiently small so that 
\begin{align*}
&\max_{i\in\{1,\dots,k\}}\Big(\theta^{\frac{2-\alpha}{\alpha}}\Big[\L^{2}(U_{0},B(x_{i},(1-\eps)r)\Big]^{1/\alpha}\Big)^{m}\\
&\quad\le\Big(\theta^{\frac{2-\alpha}{\alpha}}\sup_{y\in U_{\text{T}}}\Big[\L^{2}(U_{0},B(y,r))\Big]^{1/\alpha}\Big)^{m}+\eta.
\end{align*}
Since $\eta>0$ is arbitrary, the proof is complete.
\end{proof}

%%%%%%%%%%%%%%%%%%%%%%%%%%%%%%%%%%%%%%%%%%%%%%%%%%%%%%%%%%%%%%%%%%%%%%%%%%%%%%%%%%%%%%%%%%%%%%%%%%%%%%%%%%%%%%%%%%%%%%%%
\begin{lemma}\label{comb}
If $\{l_{i}\}_{i=1}^{k}$ is any set of $k\ge1$ real numbers, then
\begin{align}\label{eq:comb0}
\sum_{j=1}^{k} \Big((-1)^{j-1}\sum_{I\subseteq\{1,\ldots,k\}\atop |I|=j} \min_{i\in I}l_{i}\Big)
=\max_{i\in\{1,\dots,k\}}l_{i}.
\end{align}
\end{lemma}
%%%%%%%%%%%%%%%%%%%%%%%%%%%%%%%%%%%%%%%%%%%%%%%%%%%%%%%%%%%%%%%%%%%%%%%%%%%%%%%%%%%%%%%%%%%%%%%%%%%%%%%%%%%%%%%%%%%%%%%%

%%%%%%%%%%%%%%%%%%%%%%%%%%%%%%%%%%%%%%%%%%%%%%%%%%%%%%%%%%%%%%%%%%%%%%%%%%%%%%%%%%%%%%%%%%%%%%%%%%%%%%%%%%%%%%%%%%%%%%%%
\begin{proof}[Proof of Lemma~\ref{comb}]
Without loss of generality, assume $l_{1}\le l_{2}\le\dots\le l_{k}$. For any $j\in\{1,\dots,k\}$, we thus have
\begin{align}\label{eq:comb}
\sum_{I\subseteq\{1,\ldots,k\}\atop |I|=j}\min_{i\in I}l_{i}
=\sum_{i=1}^{k-j+1}{{k-i}\choose j-1}l_{i}.
\end{align}
To derive \eqref{eq:comb}, first observe that the smallest element of a subset of $\{l_1,\dots,l_k\}$ of size $j\in\{1,\dots,k\}$ can be $l_i$ for any $i\in\{1,\dots,k-j+1\}$. Further, there are ${k-i\choose j-1}$ such subsets since having chosen index $i\in\{1,\dots,k-j+1\}$, one must choose the other $j-1$ elements from the $k-i$ elements which are not smaller than $l_i$.

Hence, upon changing the order of summation in \eqref{eq:comb0}, it remains to show
\begin{align*}
\sum_{i=1}^{k}l_{i}\sum_{j=1}^{k-i+1}(-1)^{j-1}{k-i\choose j-1}=l_{k}.
\end{align*}
Since $l_{1}\le l_{2}\le\dots\le l_{k}$ are arbitrary, it thus remains to show
\begin{align*}
\sum_{j=1}^{k-i+1}(-1)^{j-1}{k-i\choose j-1}
=\begin{cases}
1 & \text{if }i=k,\\
0 & \text{if }i\in\{1,\dots,k-1\}.
\end{cases}
\end{align*}
The case $i=k$ is immediate. The case $i\in\{1,\dots,k-1\}$ is a special case of Corollary 2 in \cite{ruiz1996}.
\end{proof}
%%%%%%%%%%%%%%%%%%%%%%%%%%%%%%%%%%%%%%%%%%%%%%%%%%%%%%%%%%%%%%%%%%%%%%%%%%%%%%%%%%%%%%%%%%%%%%%%%%%%%%%%%%%%%%%%%%%%%%%%

%%%%%%%%%%%%%%%%%%%%%%%%%%%%%%%%%%%%%%%%%%%%%%%%%%%%%%%%%%%%%%%%%%%%%%%%%%%%%%%%%%%%%%%%%%%%%%%%%%%%%%%%%%%%%%%%%%%%%%%%%%%%%%%%%%%%%%%%%%%%%%%%%%%%%%%%%%%%%%%%%%%%%%%%%%%%%%%%%%%%%%%%%%%%%%%%%%%%%%%%%%%%%%%%%%%%%%%%%%%%%%
\section{Examples}\label{examples}

We now illustrate Theorems~\ref{main} and \ref{mainsub} in several examples.

%%%%%%%%%%%%%%%%%%%%%%%%%%%%%%%%%%%%%%%%%%%%%%%%%%%%%%%%%%%%%%%%%%%%%%%%%%%%%%%%%%%%%%%%%%%%%%%%%%%%%%%%%%%%
\subsection{Pure diffusion on a $d$-dimensional torus}\label{extorus}

Suppose the searchers are pure diffusion processes with diffusivity $D>0$ on the $d$-dimensional torus with diameter $l=\sup_{x,y\in M}\|x-y\|>0$. That is, the searchers diffuse on
\begin{align}\label{torusexample}
M=\big[0,{2l/\sqrt{d}}\big)^{d}\subset\R^{d}
\end{align}
with periodic boundary conditions and suppose the searchers all start at the ``center'' of the torus,
\begin{align}\label{center}
U_{0}=x_{0}={(l/\sqrt{d},\dots,l/\sqrt{d})}\in M\subset\R^{d}.
\end{align}
Suppose the metric $\L_{B}$ in \eqref{ball} is the standard Euclidean metric and the target is the entire torus, $U_{\text{T}}=M$. To avoid trivial cases, assume the searcher detection radius does not cover the entire torus (i.e. $r<l$ so that \eqref{nontrivial0} is satisfied).

%%%%%%%%%%%%%%%%%%%%%%%%%%%%%%%%%
\begin{figure}[t]
  \centering
             \includegraphics[width=1\textwidth]{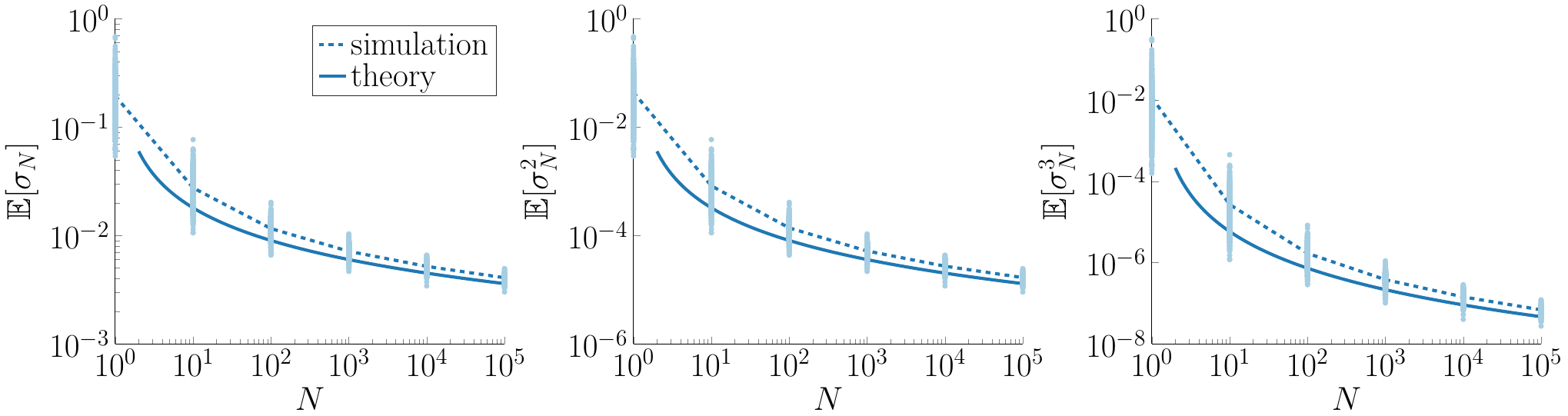}             
    \caption{Comparison of stochastic simulations (dashed curve) on the $(d=2)$-dimensional torus and the moment formula (solid curve) in \eqref{ex1formula} for the $m=1$ moment (left panel), $m=2$ moment (middle panel) and $m=3$ moment (right panel). The circle markers at $N=10^{0},10^{1},\dots,10^{5}$ are scatter plots of individual stochastic realizations of the diffusive cover time $\sigma_{N}$. The concentration of these scatter plots at the mean for large $N$ illustrates that the cover time becomes deterministic as $N\to\infty$ (see \eqref{cv}). The target is the entire torus with diameter $l=1/\sqrt{2}$ (i.e.\ the torus is a square with side length 1 and periodic boundary conditions), the detection radius is $r=0.3$, and the diffusion coefficient is $D=1$.}
 \label{figsim2d}
\end{figure}
%%%%%%%%%%%%%%%%%%%%%%%%%%%%%%%%%

Now, it is immediate that \eqref{finite} holds with $N=1$. Further, Theorem~1.2 in \cite{norris1997} implies that \eqref{log} holds with $\L$ given by the standard Euclidean metric,
\begin{align*}
\lim_{t\to0+}t\ln\P(\tau(U)\le t)
=-\frac{\sup_{x\in U}\|x_{0}-x\|^{2}}{4D}\le0.
\end{align*}
Hence, the moment formula \eqref{mainformula} in Theorem~\ref{main} holds with $\L$ given by the standard Euclidean metric. In particular, 
\begin{align}\label{ex1formula}
\E[(\sigma_{N})^{m}]
\sim
\Big(\frac{(l-r)^{2}}{4D\ln N}\Big)^{m}
\quad\text{as }N\to\infty,
\end{align}
since
\begin{align}\label{ex1L}
\sup_{y\in U_{\text{T}}}\L(x_{0},B(y,r))=l-r>0.
\end{align}
As noted in the Introduction, \eqref{ex1formula} is counterintuitive in that it (i) is independent of the space dimension $d\ge1$, (ii) depends only weakly on the detection radius $r>0$, (iii) depends on the geometry of the target set $U_{\text{T}}$ only through \eqref{ex1L}. In particular, \eqref{ex1formula} is unchanged if the target set is either the entire torus, $U_{\text{T}}=M$, or a single point at the origin,
\begin{align*}
U_{\text{T}}
=(0,\dots,0)\in M\subset\R^{d}.
\end{align*}

%%%%%%%%%%%%%%%%%%%%%%%%%%%%%%%%%%%%%%%%%%%%%%%%%%%%%%%%%%%%%%%%%%%%%%%%%%%%%%%%%%%%%%%%%%%%%%%%%%%%%%%%%%%%
\subsection{Stochastics simulations on the torus in dimensions $d=1$ and $d=2$}\label{exsim}

In Figure~\ref{figsim2d}, we compare the moments of $\sigma_{N}$ computed from stochastic simulations with the asymptotic result in \eqref{ex1formula} from Theorem~\ref{main} for diffusion on a $(d=2)$-dimensional torus. The stochastic simulation algorithm is detailed in the Appendix. This plot illustrates excellent agreement between the simulations and the asymptotic theory. This plot also illustrates how the cover time becomes deterministic for many searchers, as the blue circle markers (which show individual stochastic realizations of $\sigma_{N}$) concentrate at their mean for large $N$ (see also \eqref{cv}).

%%%%%%%%%%%%%%%%%%%%%%%%%%%%%%%%%
\begin{figure}[t]
  \centering
             \includegraphics[width=.49\textwidth]{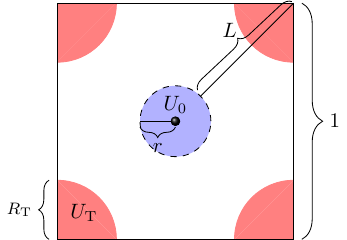}  
             \includegraphics[width=.49\textwidth]{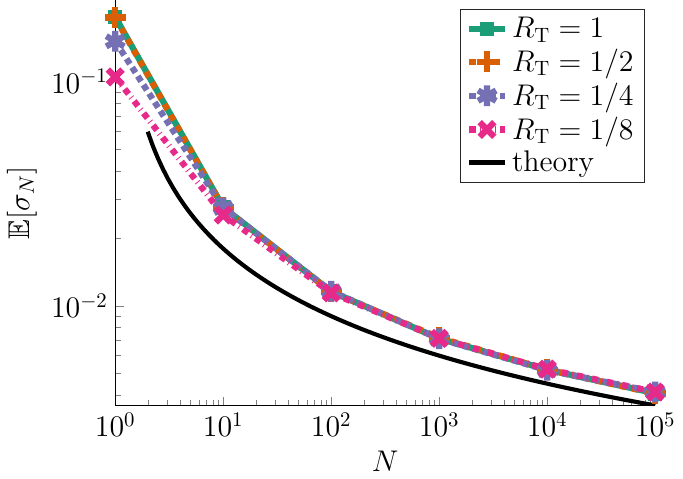}
    \caption{The left panel is a diagram showing that the searchers start in the ``center'' of the $(d=2)$-dimensional torus and the target is a disk of radius $R_{\text{T}}>0$ occupying the ``corners'' of the torus. The right panel compares stochastic simulations of the mean diffusive cover time for different choices of the target radius $R_{\text{T}}$. In agreement with theory in \eqref{ex1formula}, the cover time is independent of the target radius for large $N$.}
 \label{figsimr}
\end{figure}
%%%%%%%%%%%%%%%%%%%%%%%%%%%%%%%%%

In Figure~\ref{figsimr}, we illustrate how the cover time $\sigma_{N}$ is independent of the size of the target for large $N$ (assuming a fixed $L=\sup_{y\in U_{\text{T}}}\L(U_{0},B(y,r))$). Specifically, we plot the moments of $\sigma_{N}$ for the $(d=2)$-dimensional torus in \eqref{torusexample} where the searchers start at the ``center'' in \eqref{center} and the target is the disk centered $(0,0)\in M\subset\R^{2}$ with radius $R_{\text{T}}>0$ (see Figure~\ref{figsimr} for an illustration). For the largest value of $R_{\text{T}}$ in this plot, the target is the entire domain, and for the smallest value of $R_{\text{T}}$ in this plot, the target occupies less than $5\%$ of the area of the domain. Nevertheless, the cover times for these vastly different target areas are nearly indistinguishable once the number of searcher exceeds about $N=10$. This point is further illustrated in Figure~\ref{figsimr2}, where we plot the moments of $\sigma_N$ in the case that the target consists of several disks of common radius $R_{\text{T}}>0$ placed irregularly on the torus.

%%%%%%%%%%%%%%%%%%%%%%%%%%%%%%%%%
\begin{figure}[t]
  \centering
             \includegraphics[width=.49\textwidth]{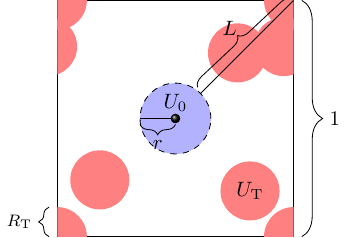}  
             \includegraphics[width=.49\textwidth]{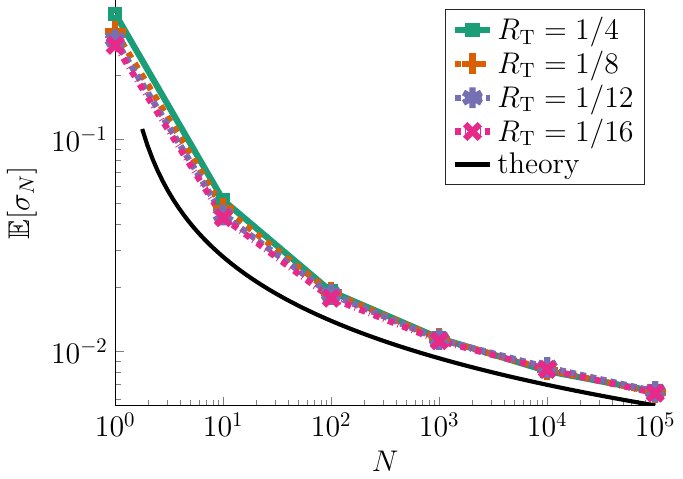}
    \caption{The left panel is a diagram showing that the searchers start in the ``center'' of the $(d=2)$-dimensional torus and the target is several irregularly placed disks of radius $R_{\text{T}}>0$. The right panel compares stochastic simulations of the mean diffusive cover time for different choices of $R_{\text{T}}$. In agreement with theory in \eqref{ex1formula}, the cover time is independent of $R_{\text{T}}$ for large $N$ since $L$ is independent of $R_{\text{T}}$.}
 \label{figsimr2}
\end{figure}
%%%%%%%%%%%%%%%%%%%%%%%%%%%%%%%%%

In Figure~\ref{figsim2d1d}, we illustrate that the cover time $\sigma_{N}$ is independent of the space dimension $d\ge1$ for large $N$. Specifically, Figure~\ref{figsim2d1d} plots the ratio of the mean cover times for a torus in dimensions $d=1$ and $d=2$ (i.e.\ $\E[\sigma_{N}^{(d=1)}]/\E[\sigma_{N}^{(d=2)}]$). As implied by \eqref{ex1formula} from Theorem~\ref{main}, this ratio approaches unity as $N$ grows.

%%%%%%%%%%%%%%%%%%%%%%%%%%%%%%%%%%%%%%%%%%%%%%%%%%%%%%%%%%%%%%%%%%%%%%%%%%%%%%%%%%%%%%%%%%%%%%%%%%%%%%%%%%%%
\subsection{Subdiffusion}\label{exsub}

We now illustrate how Theorem~\ref{mainsub} can be applied to subdiffusive motion. A common model of subdiffusion is a time fractional diffusion equation or Fokker-Planck equation of the following form \cite{metzler1999},
\begin{align}\label{ffpe}
\frac{\partial}{\partial t}q
=\D {{D}}\Delta q,\quad x\in\R^{d},
\end{align}
where $q=q(x,t)$ denotes the probability density of the subdiffusive searcher,
\begin{align*}
q(x,t)\,\dd x
=\P(Y(t)=\dd x),
\end{align*}
$\D$ denotes the Riemann-Liouville time-fractional derivative \cite{samko1993} defined by
\begin{align*}
\D q(x,t)
=\frac{\partial}{\partial t}\int_{0}^{t}\frac{1}{\Gamma(\alpha)(t-t')^{1-\alpha}}q(x,t')\,\dd t',\quad \alpha\in(0,1),
\end{align*}
a ${{D}}>0$ is the subdiffusion coefficient (with dimension $(\text{length})^{2}(\text{time})^{-\alpha}$). Note that $\D$ is often denoted by $\DD$ \cite{metzler1999}. The stochastic path of such subdiffusive searchers whose probability densities satisfy \eqref{ffpe} can be obtained via a random time change or ``subordination'' of a normal diffusive searcher \cite{magdziarz2016}. More precisely, suppose $X=\{X(s)\}_{s\ge0}$ denotes the path of a normal diffusive searcher,
\begin{align}\label{XK}
\dd X(s)
=\sqrt{2{{D}}}\,\dd W(s),
\end{align}
where $W=\{W(s)\}_{s\ge0}$ is a standard Brownian motion (note that $X$ and $W$ indexed by $s\ge0$ which has dimension $[s]=(\text{time})^{-\alpha}$). The path of the subdiffusive searcher can then be obtained as
\begin{align}\label{timechange}
Y(t)
=X(S(t)),\quad t\ge0,
\end{align}
where $\{S(t)\}_{t\ge0}$ is an inverse $\alpha$-stable subordinator that is independent of $X$ \cite{magdziarz2016}.

%%%%%%%%%%%%%%%%%%%%%%%%%%%%%%%%%
\begin{figure}[t]
  \centering
             \includegraphics[width=.49\textwidth]{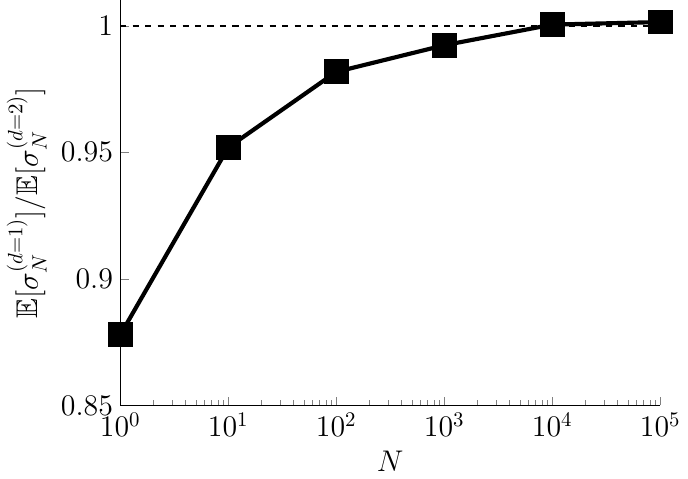}             
             \includegraphics[width=.49\textwidth]{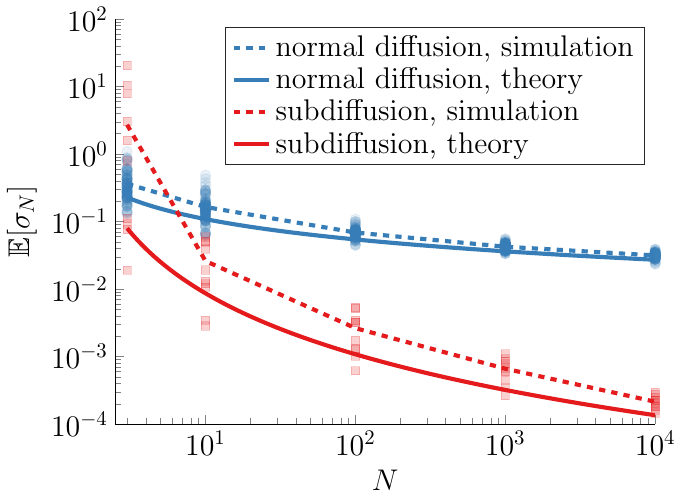}                          
             
    \caption{The left panel plots results from stochastic simulations of the diffusive cover time of a $(d=1)$-dimensional torus and a $(d=2)$-dimensional torus. In agreement with the theory in \eqref{ex1formula}, the cover time is independent of the space dimension $d$ for large $N$. The right panel compares the cover time of the $(d=1)$-dimensional torus for normal diffusion (blue) versus subdiffusion (red). In agreement with the theory, cover times are faster for subdiffusion than for normal diffusion for large $N$.  The circle and square markers are scatter plots of individual stochastic realizations of the cover time $\sigma_{N}$. In this plot, we take $L=D=1$.
    }
 \label{figsim2d1d}
\end{figure}
%%%%%%%%%%%%%%%%%%%%%%%%%%%%%%%%%

For our purposes, the important point is that the short-time distribution of FPTs of $Y$ can be obtained via the short-time distribution of $X$ using the relation \eqref{timechange}. In particular, Corollary~3 in \cite{lawley2020sub} implies that if FPTs of any process $X$ satisfy \eqref{log}, then FPTs of any process $Y$ defined via the time change of $X$ in \eqref{timechange} must satisfy \eqref{logsub}. Therefore, any application of Theorem~\ref{main} to study the cover times of a diffusive process $X$ can be generalized to an application of Theorem~\ref{mainsub} to a subdiffusive process $Y$ defined as a time change of $X$ in \eqref{timechange}, assuming merely that \eqref{finite} is satisfied for FPTs of the subdiffusive search process. We note that Theorem~8 in \cite{lawley2020sub} ensures that \eqref{finite} is satisfied as long as FPTs of the normal diffusive search process $X$ have survival probabilities that decay no slower than algebraically at large time.

In Figure~\ref{figsim2d1d}, we plot that moments of the cover time $\sigma_{N}$ for a subdiffusive search on the $(d=1)$-dimensional torus as in sections~\ref{extorus}-\ref{exsim} above (i.e.\ the domain is an interval with periodic boundary conditions). The paths of the subdiffusive searchers $\{Y_{n}(t)\}_{n=1}^{N}$ are iid realizations of $Y$ defined in \eqref{timechange} where $X$ satisfies \eqref{XK}. The stochastic simulation algorithm is detailed in the Appendix. In this plot, we take the subdiffusive exponent $\alpha=1/2$ and the target is the entire domain. As a technical point, we note that \eqref{finite} must hold for subdiffusion in this example since survival probabilities of FPTs of normal diffusion decay exponentially at large time (due to the finite domain), and thus Theorem~8 in \cite{lawley2020sub} implies that \eqref{finite} holds for subdiffusion for all $N>1/\alpha$. 

Figure~\ref{figsim2d1d} illustrates excellent agreement between the simulations and the asymptotic theory. Furthermore, Figure~\ref{figsim2d1d} illustrates that subdiffusive cover times are actually faster than normal diffusive cover times for many searchers (assuming subdiffusion is modeled by fractional Fokker-Planck equation). Indeed, Theorem~\ref{main} implies that diffusive cover times decay as $1/\log N$ whereas Theorem~\ref{mainsub} implies that subdiffusive cover times decay as $1/\log N^{2/\alpha-1}$ where $2/\alpha-1>1$ since $\alpha\in(0,1)$. This is shown in Figure~\ref{figsim2d1d} where the blue curves and markers show cover times for normal diffusion and the red curves and markers show cover times for subdiffusion.

To understand why cover times for many subdiffusive searchers are faster than cover times for many diffusive searchers, note that (i) both depend on rare events in which searchers move very quickly away from their initial position and (ii) subdiffusive searchers (modeled by a time-fractional Fokker-Planck equation) are more likely to move a long distance quickly than their normal diffusive counterpart. Point (ii) is most easily illustrated by comparing the normal diffusive and subdiffusive propagators in free space. Specifically, the probability density that a one-dimensional diffusive particle starting from the origin is at position $x$ after time $t>0$ is
\begin{align*}
    p(x,t)
    =\frac{1}{\sqrt{4\pi Dt}}e^{-x^2/(4Dt)},
\end{align*}
whereas this density for a subdiffusive particle decays as
\begin{align*}
    p_\alpha(x,t)
    \approx e^{-(t_{\alpha}/t)^{\alpha/(2-\alpha)}}\quad\text{for }|x|\gg\sqrt{Dt^\alpha},
\end{align*}
where we have defined the timescale $t_{\alpha}=\big(\alpha^{\alpha}(2-\alpha)^{2-\alpha}\frac{|x|^{2}}{4D}\big)^{1/\alpha}>0$ (the approximate equality omits the subdominant prefactors, see equation~(45) in \cite{metzler2000} for the full expansion). Point (ii) follows from observing that $p_\alpha(x,t)\gg p(x,t)$ for small $t$ if $\alpha\in(0,1)$.

%%%%%%%%%%%%%%%%%%%%%%%%%%%%%%%%%%%%%%%%%%%%%%%%%%%%%%%%%%%%%%%%%%%%%%%%%%%%%%%%%%%%%%%%%%%%%%%%%%%%%%%%%%%%
\subsection{Diffusion with space-dependent diffusivity and drift}\label{exdrift}

We now illustrate how Theorem~\ref{main} applies to more general drift-diffusion processes. Assume the searchers diffuse according to an It\^{o} stochastic differential equation (SDE) on $M=\R^{{d}}$,
\begin{align}\label{sde}
\begin{split}
\dd X_{n}
&={{{\mu}}}(X_{n})\,\dd t+\sqrt{2D}{\sigma}(X_{n})\,\dd W_{n},
\end{split}
\end{align}
where ${{{\mu}}}:\R^{{d}}\to\R^{{d}}$ is a possibly space-dependent drift, $D>0$ is a characteristic diffusion coefficient, ${\sigma}:\R^{{d}}\to\R^{{{d}}\times q}$ is a dimensionless, matrix-valued function that describes possible space-dependence or anisotropy in the diffusion, and $W_{n}(t)\in\R^{q}$ is a standard Brownian motion in $q$-dimensional space. As technicalities, we follow \cite{lawley2020uni} and assume ${{{\mu}}}$ is uniformly bounded and uniformly H\"{o}lder continuous and ${\sigma}{\sigma}^{\top}$ is uniformly H\"{o}lder continuous and its eigenvalues are in a finite interval bounded above zero. For a smooth path $\omega:[0,1]\to M$, define its length, $l(\omega)$, in the following Riemannian metric which depends on the diffusivity matrix $a:={\sigma}{\sigma}^{\top}$ in \eqref{sde},
\begin{align}\label{ll}
l(\omega)
:=\int_{0}^{1}\sqrt{\dot{\omega}^{\top}(s)a^{-1}(\omega(s))\dot{\omega}(s)}\,\dd s.
\end{align}
For any pair $x_{0},x\in\R^{d}$, define the geodesic to be the following infimum of over smooth paths $\omega:[0,1]\to M$ connecting $\omega(0)={x_{0}}$ to $\omega(1)={{x}}$:
\begin{align}\label{drie}
\begin{split}
\drie({x_{0}},{{x}})
&:=\inf\{l(\omega):\omega(0)=x_{0},\,\omega(1)=x\},\quad x_{0},x\in\R^{d}.
\end{split}
\end{align}

Under these assumption, Varadhan's formula \cite{varadhan1967} was used in \cite{lawley2020uni} to show that \eqref{log} holds with metric in $\L=\drie$ in \eqref{drie}. Further, assume that the drift $\mu$ and diffusivity $\sigma$ are such that \eqref{finite} holds. One simple way to ensure that \eqref{finite} holds for $N=1$ is to assume that at all sufficiently large radii, the drift $\mu$ is a constant force pointing toward the origin and $\sigma$ is the $d\times d$ identity matrix.

Hence, the moment formula \eqref{mainformula} in Theorem~\ref{main} holds with $\L=\drie$ for any compact target set $U_{\text{T}}\subset\R^{d}$. Notice that the many searcher cover time is independent of the drift in $\mu$ in the SDE \eqref{sde}. Notice also that the space-dependent diffusivity $\sigma$ affects the many searcher cover time through the geodesic $\drie$. Also, as in the simple example in sections~\ref{extorus}-\ref{exsim} above, the many searcher cover time depends only weakly on the detection radius $r>0$ and depends on the target only through the single length $L$ and is otherwise independent of the target size.

%%%%%%%%%%%%%%%%%%%%%%%%%%%%%%%%%%%%%%%%%%%%%%%%%%%%%%%%%%%%%%%%%%%%%%%%%%%%%%%%%%%%%%%%%%%%%%%%%%%%%%%%%%%%
\subsection{Diffusion on a manifold}

To illustrate another scenario in which Theorem~\ref{main} applies, assume $M$ is a ${{d}}$-dimensional smooth Riemannian manifold. Assume the $N$ searchers are iid realizations of a searcher $\{X(t)\}_{t\ge0}$ which is a diffusion on $M$ which is described by its generator $\mathcal{A}$, which in each coordinate chart is a second order differential operator of the following form
\begin{align*}
\mathcal{A} f
=D\sum_{i,j=1}^{n}\frac{\partial}{\partial x_{i}}\bigg(a_{ij}(x)\frac{\partial f}{\partial x_{j}}\bigg),
\end{align*}
where $a=\{a_{ij}\}_{i,j=1}^{n}$ satisfies certain technical assumptions (in each coordinate chart, assume $a$ is continuous, symmetric, and that its eigenvalues are in a finite interval bounded above zero). Assume $X$ reflects from the boundary of $M$ if $M$ has a boundary and assume $M$ is connected and compact to ensure that \eqref{finite} is satisfied.

Theorem~1.2 in \cite{norris1997} implies that \eqref{log} holds with $\L=\drie$ in \eqref{drie}. Hence, the moment formula \eqref{mainformula} in Theorem~\ref{main} holds with $\L=\drie$ in \eqref{drie}.

%%%%%%%%%%%%%%%%%%%%%%%%%%%%%%%%%%%%%%%%%%%%%%%%%%%%%%%%%%%%%%%%%%%%%%%%%%%%%%%%%%%%%%%%%%%%%%%%%%%%%%%%%%%%%%%%%%%%%%%%%%%%%%%%%%%%%%%%%%%%%%%%%%%%%%%%%%%%%%%%%%%%%%%%%%%%%%%%%%%%%%%%%%%%%%%%%%%%%%%%%%%%%%%%%%%%%%%%%%%%%%
\section{Discussion}

In this paper, we proved a simple formula for all the moments of cover times of many diffusive or subdiffusive searchers that applies in a great variety of mathematical models of stochastic search. This formula shows that the only relevant parameters in the many searcher limit is (i) the searcher's characteristic (sub)diffusivity and (ii) a certain geodesic distance between the searcher starting location(s) and the farthest point in the target. We illustrated these general mathematical results in several examples and using stochastic simulations. 

As detailed in the Introduction, there is a vast literature on cover times, and the majority of this prior work considers a single stochastic searcher in the parameter regime in which the cover time diverges. In a continuous space setting, this means the target to be covered is much larger than the detection radius of the searcher. For search on a discrete network, this means the target consists of a large number of discrete nodes \cite{chupeau2015}. A notable exception is the very interesting work of Majumdar, Sabhapandit, and Schehr \cite{majumdar2016} which studied the time for $N\ge1$ purely diffusive, one-dimensional searchers to cover a finite interval with either reflecting or periodic boundary conditions. In the many searcher limit $N\to\infty$, these authors obtained the moment formula in \eqref{main0}. We have thus shown that this simple one-dimensional result for pure diffusion (i.e.\ Brownian motion) extends to arbitrary, higher-dimensional spatial domains and much more general stochastic search dynamics.

It has been claimed that the cover time for $N\ge1$ searchers is a simple rescaling of the cover time for a single searcher \cite{chupeau2015, dong2023}. We have shown how this approximation breaks down as $N$ grows. Nevertheless, this investigation raises the following question: when is a particular system in the ``small detection radius $r$'' regime versus the ``many searcher $N$'' regime? That is, notice that for any fixed number of searchers $N$ (even large), the cover time $\sigma_{N}$ must diverge as the detection radius $r$ vanishes,
\begin{align}\label{rsmall}
\sigma_{N}\to\infty\quad\text{as }r\to0.
\end{align}
On the other hand, for any fixed detection radius $r>0$ (even small) we have shown that the cover time $\sigma_{N}$ must vanish as the number of searchers $N$ grows,
\begin{align}\label{Nbig}
\sigma_{N}\to0\quad\text{as }N\to\infty.
\end{align}
We conjecture that the regimes in \eqref{rsmall} and \eqref{Nbig} can be distinguished by comparing the mean cover time for a single searcher, $\E[\sigma_{1}]$ with the formula $\frac{L^{2}}{4D\ln N}$ as in \eqref{main0}. In particular, we conjecture that
\begin{align}\label{conjecture}
\E[\sigma_{N}]
\approx\max\Big\{\frac{\E[\sigma_{1}]}{N},\frac{L^{2}}{4D\ln N}\Big\},
\end{align}
so that ``small detection radius $r$'' regime of \eqref{rsmall} is guaranteed if $\frac{\E[\sigma_{1}]}{N}\gg\frac{L^{2}}{4D\ln N}$, whereas the ``many searcher $N$'' regime is guaranteed by $\frac{\E[\sigma_{1}]}{N}\ll\frac{L^{2}}{4D\ln N}$. Investigating the conjecture in \eqref{conjecture} is an interesting avenue for future work. We note that numerical investigation of \eqref{conjecture} is challenging due to the difficulty in sampling $\sigma_{N}$ for a small detection radius \cite{mendoncca2011, grassberger2017}.

An interesting avenue for future research is to investigate cover times of many searchers for other modes of stochastic search. For example, it is natural to ask about many random walkers on a discrete network \cite{lawley2020networks}, since many prior works on cover times investigate such discrete search scenarios \cite{nemirovsky1990, yokoi1990, brummelhuis1991, hemmer1998, nascimento2001, zlatanov2009, mendoncca2011, chupeau2014, chupeau2015, grassberger2017, cheng2018, maziya2020, dong2023, han2023, aldous1983, aldous1989, broder1989, aldous1989b, kahn1989, ding2012, belius2013}. In addition, given the results of the present paper, it is natural to ask about cover times of superdiffusive search. Such an investigation would depend on the model of superdiffusive search. One possible approach would be to consider cover times of L{\'e}vy flights, since the extreme first passage time distributions of such superdiffusive searchers has been worked out \cite{lawley2023super}.

%%%%%%%%%%%%%%%%%%%%%%%%%%%%%%%%%%%%%%%%%%%%%%%%%%%%%%%%%%%%%%%%%%%%%%%%%%%%%%%%%%%%%%%%%%%%%%%%%%%%%%%%%%%%%%%%%%%%%%%%%%%%%%%%%%%%%%%%%%%%%%%%%%%%%%%%%%%%%%%%%%%%%%%%%%%%%%%%%%%%%%%%%%%%%%%%%%%%%%%%%%%%%%%%%%%%%%%%%%%%%%
\section{Appendix: Stochastic simulations algorithms}

In this Appendix, we describe the stochastic simulation algorithms used in section~\ref{examples}.

%%%%%%%%%%%%%%%%%%%%%%%%%%%%%%%%%%%%%%%%%%%%%%%%%%%%%%%%
\subsection{Normal diffusion on $(d=2)$-dimensional torus}

We discretize the domain $M$ to a square lattice $\widetilde{M}$ with minimum distance $\Delta x$ and approximate the cover time over the lattice by
\begin{equation} \label{sigma_lattice}
	\sigma_N \approx \inf \{ t > 0 : \widetilde{\mathcal{S}}_N(t) \subseteq \widetilde{U}_T \},
\end{equation}
where $\widetilde{U}_T = U_T \cap \widetilde{M}$ and $\widetilde{\mathcal{S}}_N(t) = \mathcal{S}_N(t) \cap \widetilde{M}$. Using this approximation, we sample $\sigma_N$ by simulating the $N$ diffusive paths $\{X_n\}_{n=1}^N$ with unit diffusivity via a standard Euler-Maruyama scheme \cite{kloeden1992} with discrete time step $\Delta t$ until the condition \eqref{sigma_lattice} is satisfied. We choose $\Delta x \le r/10$ so that the lattice approximates the detection area well and $\Delta t \le \Delta x^2 /8$ so that most of the diffusion step sizes are under $\Delta x$.

%%%%%%%%%%%%%%%%%%%%%%%%%%%%%%%%%%%%%%%%%%%%%%%%%%%%%%%%
\subsection{Normal diffusion on $(d=1)$-dimensional torus}\label{1dsim}

Let the spatial domain $M$ be the one-dimensional torus with diameter $l>0$ (i.e.\ the interval $(-l,l)$ with periodic boundary conditions) and suppose the searchers move by pure diffusion at all start at the origin. In this simple one-dimensional setting, the cover time can be written as
\begin{align}\label{sigma1d}
\sigma_{N}
=\inf\Big\{t>0:\max_{1\le n\le N,0\le s\le t}\{X_{n}(s)\}-\min_{1\le n\le N,0\le s\le t}\{X_{n}(s)\}>2(l-r)\Big\},
\end{align}
where $\{X_{n}\}_{n=	1}^{N}$ are $N\ge1$ iid diffusion processes on the entire real line. Using the representation \eqref{sigma1d}, we sample $\sigma_{N}$ by simulating the $N$ paths $\{X_{n}\}_{n=1}^{N}$ via a standard Euler-Maruyama scheme \cite{kloeden1992} with discrete time step $\Delta t=10^{-6}$ until the condition in \eqref{sigma1d} is satisfied.

%%%%%%%%%%%%%%%%%%%%%%%%%%%%%%%%%%%%%%%%%%%%%%%%%%%%%%%%
\subsection{Subdiffusion on $(d=1)$-dimensional torus}

The subdiffusion simulations in section~\ref{exsub} follow along similar lines as the normal diffusion simulations described in section~\ref{1dsim}, except the subdiffusive paths $Y_{n}=Y_{n}(t)$ for $n=1,\dots,N$ are more complicated to simulate. Specifically, we first simulate an $\alpha$-stable subordinator $T=T(s)$ on a discrete time grid $\{s_{k}\}_{k}$ for $s_{k}=k\Delta s$ for some $\Delta s>0$ following the method of Magdziarz et al.\ \cite{magdziarz2007}. In particular, $T$ is exactly simulated on the discrete grid $\{s_{k}\}_{k}$ according to
\begin{align*}
T(s_{k+1})
=T(s_{k})+(\Delta s)^{1/\alpha}\Theta_{k},\quad k\ge0,
\end{align*}
where $T(s_{0})=T(0)=0$ and $\{\Theta_{k}\}_{k\in\mathbb{N}}$ is an independent and identically distributed sequence of realizations of 
\begin{align*}
\Theta
=\frac{\sin(\alpha(V+\pi/2)}{(\cos(V))^{1/\alpha}}\bigg(\frac{\cos(V-\alpha(V+\pi/2))}{E}\bigg)^{\frac{1-\alpha}{\alpha}},
\end{align*}
where $V$ is uniformly distributed on $(-\pi/2,\pi/2)$ and $E$ is an independent unit rate exponential random variable. Having sampled $\{T(s_{k})\}_{k}$, we approximate the inverse subordinator $S=\{S(t)\}_{t\ge0}$ defined by
\begin{align*}
S(t)
:=\inf\{s>0:T(s)>t\}
\end{align*}
on a discrete time grid $\{t_{m}\}_{m}$ with $t_{m}=m\Delta t$ for some $\Delta t>0$. In particular, we set $S(t_{m})=s_{k}$ where $k$ is the unique index such that $T(s_{k-1})<t_{m}\le T(s_{k})$. Finally, we obtain $Y$ on the discrete time grid $\{t_{m}\}_{m}$ via linear interpolation,
\begin{align*}
Y(t_{m})
=\Big(\frac{S(t_{m})-s_{k}}{s_{k+1}-s_{k}}\Big)X(s_{k+1})
+\Big(\frac{s_{k+1}-S(t_{m})}{s_{k+1}-s_{k}}\Big)X(s_{k}),\quad m\ge1,
\end{align*}
where $k$ is the largest index such that $s_{k}\le S(t_{m})\le s_{k+1}$ and $\{X(s_{k})\}_{k}$ is simulated via a standard Euler-Maruyama scheme \cite{kloeden1992}. We take $\Delta s=\Delta t=10^{-2}/(4\ln N)$ for the simulations in section~\ref{exsub}.

%%%%%%%%%%%%%%%%%%%%%%%%%%%%%%%%%%%%%%%%%%%%%%%%%%%%%%%%%%%%%%%%%%%%%%%%%%%%%%%%%%%%%%%%%%%%%%%%%%%%%%%%%%%%%%%%%%%%%%%%%%%%%%%%%%%%%%%%%%%%%%%%%%%%%%%%%%%%%%%%%%%%%%%%%%%%%%%%%%%%%%%%%%%%%%%%%%%%%%%%%%%%%%%%%%%%%%%%%%%%%%%%%%%%%%%%%%%%%%%%%%%%%%%%%%%

\bibliography{library.bib}
\bibliographystyle{siam}

\end{document}